\documentclass[12pt]{amsart}
\input epsf.tex

\usepackage{amsmath}
\usepackage{amsthm}
\usepackage{amsfonts}
\usepackage{amssymb}
\usepackage{graphicx}
\usepackage{latexsym}

\usepackage{amsmath,amsthm,amsfonts,amssymb}

\usepackage{epsfig}

\usepackage{amssymb}
\usepackage[all]{xy}
\xyoption{poly}
\usepackage{fancyhdr}
\usepackage{wrapfig}

\theoremstyle{plain}
\newtheorem{thm}{Theorem}[section]
\newtheorem{prop}[thm]{Proposition}
\newtheorem{lem}[thm]{Lemma}
\newtheorem{cor}[thm]{Corollary}

\theoremstyle{definition}
\newtheorem{defn}{Definition}
\theoremstyle{remark}
\newtheorem{remark}{Remark}

\newtheorem{notation}{Notation}

\advance \topmargin by -\headheight
\advance \topmargin by -\headsep
\evensidemargin \oddsidemargin
\def\cA{{\mathcal A}}

\def\cB{{\mathcal B}}

\def\cc{{\curvearrowright}}

\def\D{\mathbb D}

\def\E{{\mathbb E}}
\def\EE{{\mathbb E}}

\def\F{{\mathbb F}}
\def\FF{{\mathbb F}}
\def\cF{{\mathcal F}}

\def\cI{{\mathcal I}}

\def\cL{{\mathcal L}}

\def\N{{\mathbb N}}

\def\tnu{{\tilde{\nu}}}
\newcommand\norm[1]{\left\|#1\right\|}

\def\cP{{\mathcal P}}

\def\cR{\mathcal R}

\def\cS{{\mathcal S}}

\def\cT{{\mathcal T}}

\def\chix{{\raise.5ex\hbox{$\chi$}}}

\def\Z{{\mathbb Z}}

\begin{document}
\title[Mean  convergence of Markovian spherical averages]{Mean  convergence of Markovian spherical averages for measure-preserving actions of
 the free group}
\author{Lewis Bowen}\address{University of Texas at
    Austin}
\email{lpbowen@math.utexas.edu} 
\author{Alexander Bufetov}\address{Aix-Marseille Universit{\'e}, CNRS, Centrale Marseille, I2M, UMR 7373}\address{ The Steklov Institute of Mathematics, Moscow}\address{
The Institute for Information Transmission Problems, Moscow}\address{
National Research University Higher School of Economics, Moscow}
\address{
Rice University, Houston
}
\email{bufetov@mi.ras.ru} \author{Olga Romaskevich}
\address{\'{E}cole Normale Sup\'{e}rieure de Lyon}\address{National Research University Higher School of Economics, Moscow}
 \email{olga.romaskevich@ens-lyon.fr} 
\date{}
\maketitle
\begin{abstract}
Mean convergence of Markovian spherical averages is established for a
measure-preserving action of a finitely-generated free group on a probability space.  We endow the set of generators with a  generalized Markov chain and establish  the mean convergence of resulting spherical averages in this case under mild nondegeneracy assumptions on the stochastic matrix $\Pi$ defining our Markov chain. Equivalently, we establish the triviality of the tail sigma-algebra of the corresponding Markov operator. This convergence was previously known only for symmetric Markov chains, while the conditions  ensuring convergence in our paper are inequalities rather than equalities, so mean convergence of spherical averages is established for a much larger class of Markov chains.
\end{abstract}
\noindent

\noindent
\section{Introduction}

Consider a finitely generated free group $\FF$ and a probability space $(X,\mu)$. 

Let $T:\FF \to \rm{Aut}(X, \mu)$ denote a homomorphism of $\FF$ into the group of measure-preserving transformations of $(X,\mu)$. We consider a finite alphabet $V$ with a labeling map $\cL:V \to \FF$. 

We will study an arbitrary Markov chain with $V$ being its set of states. That is, take a stochastic matrix $\Pi=(\Pi_{v,w})_{v,w\in V}$ with rows and columns indexed by the elements of $V$ (so $\sum_w \Pi_{v,w} = 1$ for every $v$). We assume that $\Pi$ has a stationary distribution $\nu:V \to [0,1]$ with $\nu(v)>0$ for all $v \in V$. Stationarity means that $\sum_{v\in V} \Pi_{w,v}\nu(v)= \nu(w)$ for any $w$. 

Let $G=(V,E)$ denote the directed graph on $V$ with edge set 
$$E:=\{ (w,v):~\Pi_{vw}>0\}.$$
Note $(w,v)$ is the reverse of $(v,w)$ above. This is intentional. 

By a {\em directed path} in $G$ we mean a sequence  $s=(s_1,\ldots, s_n) \in V^n$ of vertices such that $(s_i,s_{i+1})\in E$ for all $i$. The length of such a path is $|s|:=n$. For any such path we denote
$$
\cL(s)=\cL(s_1)\cdots \cL(s_n) \in \FF, \quad T_s=T_{\cL(s)} \in \rm{Aut}(X,\mu), \quad \Pi_s=\Pi_{s_n s_{n-1}} \cdots \Pi_{s_2 s_1}.$$

Define spherical averages $S_n:L^1(X,\mu) \to L^1(X,\mu)$ by the formula 

\begin{equation}\label{eq:defn_spherical_averages}
S_n(\phi)(x):= \sum_{s=(s_1,\ldots, s_n)} \nu(s_n) \Pi_{s} \phi(T_s x)
\end{equation}

The goal of this paper is to prove that, under mild additional conditions on $\Pi$, there is a constant $k$ such that the averages $\frac{1}{2k} \sum_{i=0}^{2k-1} S_{n+i}$ are mean ergodic in $L^1$. To state these conditions properly, we need more notation.

\begin{notation}
 If $p \in V^k$ and $q\in V^l$ then we let $pq \in V^{k+l}$ be their concatenation. So if $p=(p_1,\ldots, p_k)$ and $q=(q_1,\ldots, q_l)$ then $pq= (p_1,\ldots p_k,q_1,\ldots, q_l)$. We let $\cL(p)=\cL(p_1)\cdots \cL(p_k) \in \FF$ denote the product of the labels.
\end{notation}

\begin{defn}\label{defn:good-graph}
A subgraph $H\subset G$ is {\em good} of order $k$ if it consists of vertices $u,w$ and directed paths $p,q,p^*, q^*$ of length $k$ so that
\begin{itemize}
\item $upw, uqw, pq^*p, qp^*q$ are directed paths in $G$
\item $\cL(p^*) = \cL(p)^{-1}$, $\cL(q^*) = \cL(q)^{-1}$
\end{itemize}
Figure \ref{pic_chain} illustrates the structure of a good subgraph. We do not require that a good subgraph be induced. 
\end{defn}

\begin{defn}
For each $v \in V$, let $\Gamma_v \le\FF$ be the subgroup generated by all elements of the form $\cL(p)$ where $pv$ is a directed path from $v$ to itself in $G$. To be more precise, the condition on $p$ is that it be a directed path of the form $p=(p_1,\ldots, p_n) \in V^n$ such that $p_1=v$ and $(p_n,v)\in E$ is an edge of $G$.
\end{defn}

\begin{defn}
We will say that $\Pi$ is {\em admissible} of order $k$ if
\begin{itemize}
\item its associated graph $G$ contains a good subgraph of order $k$,
\item $G$ is strongly connected and
\item there is some $v\in V$ such that $\Gamma_v=\FF$.
\end{itemize}
\end{defn}

\begin{figure}
\begin{center}
\includegraphics*[scale=0.7]{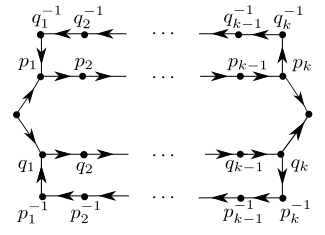}
\caption[]{A good subgraph with paths $p=(p_1, \ldots, p_k)$ and $q=(q_1, \ldots, q_k)$ from Definition \ref{defn:good-graph}. We have used the notation $p^*=(p_k^{-1},\ldots, p_1^{-1})$ and $q^*=(q_k^{-1},\ldots, q_1^{-1})$.}\label{pic_chain}
\smallskip
\end{center}
\end{figure}

\begin{thm}\label{thm:main}
Suppose $\Pi$ is admissible of order $k$. Then for any probability-measure-preserving action $\FF \cc (X,\mu)$ and any $f\in L^1(X,\mu)$
$$\frac{1}{2k} \sum_{i=0}^{2k-1} S_{n+i} f$$
converges in $L^1$ to $\E[f|\FF]$ as $n\to\infty$, where $\E[f|\FF]$ is the conditional expectation on the sigma algebra of $\FF$-invariant measurable subsets. 
\end{thm}

\begin{remark}
Note that the conditions on $\Pi$ depend only on which entries are positive and which are zero. In particular, no relations are assumed between the entries of the Markov chain.
\end{remark}

In practice, it is a straightforward task to check whether $\Pi$ is admissible. We note for example, the following special case:

\begin{prop}\label{prop:sufficient}
Suppose $V$ is finite and $\cL:V \to \F$ is injective, so that we may identify $V$ as a subset of $\F$. Also suppose $G$ is strongly connected and for every $(a,b)\in E$, $(b^{-1},a^{-1}) \in E$ where the inverse is taken in the group $\F$. If there exist $v,w,u \in V$ such that $(v,w), (u,w), (u,v^{-1}) \in E$ (see Figure \ref{pic_condition}) then $G$ contains a good subgraph. So if there is some $v\in V$ such that $\Gamma_v=\FF$ then the conclusion to Theorem \ref{thm:main} holds.
\end{prop}

\begin{figure}
\begin{center}
\includegraphics*[scale=0.7]{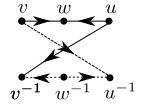}
\caption[]{Sufficient condition for a graph $G$ to contain a good subgraph, see Proposition \ref{prop:sufficient}}\label{pic_condition}
\smallskip
\end{center}
\end{figure}

\begin{proof}
Note $(v,u^{-1}),  (w^{-1}, v^{-1}) \in E$. Because $G$ is strongly connected and finite, there exists a $k$ so that for any ordered pair of vertices of $G$ there exists a directed path between them of length $k$. In particular there exists a directed path $p:=(p_1,\ldots, p_k)$ from $p_1:=w$ to $p_k:=v$ and a directed path $q:=(q_1,\ldots, q_k)$ from $q_1:=v^{-1}$ to $q_k:=u$. It is now elementary to check that $upw, uqw, pq^*p, qp^*q$ are directed paths in $G$ where $p^*$ is the unique directed path in $G$ with $\cL(p^*)=\cL(p)^{-1}$. 
\end{proof}

\subsection{Historical remarks.}

For two rotations of a sphere, convergence of spherical averages was established by Arnold and Krylov \cite{ArnKr}, and a general mean ergodic theorem for actions of free groups was proved by Guivarc'h \cite{Guivarch}.

A first general pointwise ergodic theorem for convolution averages on a countable group is due to Oseledets \cite{Oseled} who 
relied on the martingale convergence theorem.

First general pointwise ergodic theorems for free semigroups and groups were given 
by R.I. Grigorchuk in 1986 \cite{Grig86}, where the main result is Ces{\`a}ro convergence of spherical averages 
for measure-preserving actions of a  free semigroup and group.  Convergence of the spherical averages themselves was established by Nevo \cite{Nevo94} for functions in $L_2$ and Nevo and Stein \cite{NeSt} for functions in $L_p$, $p>1$ using deep spectral theory methods.
Whether uniform spherical averages of an integrable function under the action of a free group converge almost surely remains an open problem (it is tempting to speculate that a counterexample might be possible along the lines of Ornstein's example \cite{Ornst}).
The method of Markov operators in the proof of  ergodic theorems for actions of free semigroups and groups 
was suggested by R. I. Grigorchuk \cite{Grig99}, J.-P. Thouvenot (oral communication), and in \cite{Buf00}.
In \cite{Buf02} pointwise convergence is proved for Markovian spherical averages under the additional assumption that 
the Markov chain be reversible. The key step in \cite{Buf02} is the triviality of the tail sigma-algebra for the corresponding 
Markov operator; this is proved using Rota's ``Alternierende Verfahren'' \cite{Rota}, that is to say, martingale convergence. 
The reduction of powers of the Markov operator to Rota's ``Alternierende Verfahren''  in \cite{Buf02} essentially relies on the 
reversibility of the Markov chain. In this paper, we show that the triviality of the tail sigma-algebra still holds under much milder assumptions on the underlying  chain. 

The study of Markovian averages is  motivated by the problem of ergodic theorems for general countable groups, specifically, for groups admitting a Markovian coding such as Gromov hyperbolic groups \cite{Gromov} (see e.g. Ghys-de la Harpe \cite{GhysDelaharpe} for a detailed discussion of the Markovian coding for Gromov hyperbolic groups). First results on convergence of spherical averages for Gromov hyperbolic groups, obtained  under strong exponential mixing assumptions on the action, are due to Fujiwara and Nevo \cite{FuNe}. For actions of hyperbolic groups on finite spaces, an ergodic theorem was obtained by Bowen in \cite{Bowen}.

Ces{\`a}ro convergence of spherical averages for all measure-preserving actions of Markov semigroups, and, in particular, Gromov hyperbolic groups, was established by Bufetov, Klimenko and Khristoforov in \cite{BKK}. In the special case of hyperbolic groups, a short and very elegant proof of this theorem, using the method of Calegari and Fujiwara \cite{CalFuji}, was later given by Pollicott and  Sharp \cite{PS}.
Using the method of amenable equivalence relations, Bowen and Nevo \cite{BowNe4953},
\cite{BowNe4426}, \cite{BowNe4109}, \cite{BowNe3569} established ergodic theorems for 
``spherical shells''  in Gromov hyperbolic groups. The latter do not require any mixing assumptions.

\subsection{Examples}

\subsubsection{Uniform spherical averages}
Consider the special case in which $\F=\langle a_1,\ldots, a_r\rangle$ and $V=\{a_1,\ldots, a_r\} \cup \{a_1^{-1},\ldots, a_r^{-1}\} \subset \F$. We let $\cL:V \to \F$ be the inclusion map and $\Pi_{a,b} = \frac{1}{2r-1}$ if $a \ne b^{-1}$, $\Pi_{a,b}=0$ otherwise. We let $\nu$ be the stationary distribution that is uniformly distributed on $V$. In this case, $\Pi$ is admissible of order 1 and $S_n$ is the uniform average on the sphere of radius $n$ centered at the identity in $\F$. That is, 
$$S_n(\phi)(x) = |\{g\in \F:~|g|=n\}|^{-1} \sum_{|g|=n} \phi(T_g x)$$
for $\phi \in L^1(X,\mu)$ and $x \in X$. So Theorem \ref{thm:main} proves the mean ergodic theorem for the averages $\frac{S_n+S_{n+1}}{2}$. This result was first obtained by Guivarc'h \cite{Guivarch}. 

\subsubsection{A surface group example}
Let $\Lambda=\langle a,b,c,d| [a,b][c,d]=1\rangle$ denote the fundamental group of the closed genus 2 surface. There is a natural Markov coding of this group, developed by Bowen-Series \cite{bowen-series}, that was used in \cite{BufSer} to prove a pointwise ergodic theorem for Ces\`aro averages of spherical averages (with respect to the word metric on this group). Using this coding and Theorem \ref{thm:main} we will show: 
\begin{cor}\label{cor:surface}
There exists a sequence $\pi_n$ of probability measures on $\Lambda$ such that
\begin{itemize}
\item $\pi_n$ is supported on the union of the spheres of radius $n$ and radius $n+1$ centered at the identity in $\Lambda$ (with respect to the word metric);
\item $\pi_n$ is mean ergodic in $L^1$ in the sense that: if $\Lambda \cc (X,\mu)$ is any probability-measure-preserving action and $f \in L^1(X,\mu)$ then the averages $\pi_n(f) \in L^1(X,\mu)$ defined by
$$\pi_n(f)(x)=\sum_{g\in \Lambda} \pi_n(g)f(g^{-1}x)$$
converge in $L^1(X,\mu)$ to $\E[f|\Lambda]$, the conditional expectation of $f$ on the sigma-algebra of $\Lambda$-invariant subsets.
\end{itemize}
\end{cor}

To explain the coding, let $\cR$ denote a regular octogon in the hyperbolic plane (which we identify with $\D$ the unit disk in the complex plane) with all interior angles equal to $\pi/4$. This is a fundamental domain for an action of $\Lambda$ on $\D$ by isometries. It can be arranged that if $S=\{a,b,c,d,a^{-1},b^{-1},c^{-1},d^{-1}\}$ then $\cR \cap s\cR$ is an edge of $\cR$ for any $s\in S$. 

Let $\cT=\cup_{g\in \Lambda} g\partial \cR$ be the union of the boundaries of $\Lambda$-translates of $\cR$. We may think of $\cT$ as a union of bi-infinite geodesics. Let $\cP\subset \partial\D$ denote the collection of endpoints of those geodesics in $\cT$ which meet $\cR$ (crucially this includes lines which meet $\partial \cR$ only in a vertex of $\cR$). The points $\cP$ partition $\partial \D-\cP$ into connected open intervals; we denote the collection of all these intervals by $\cI$. See figure \ref{fig:surface}.

\begin{figure}
\begin{center}
\includegraphics*[scale=0.4]{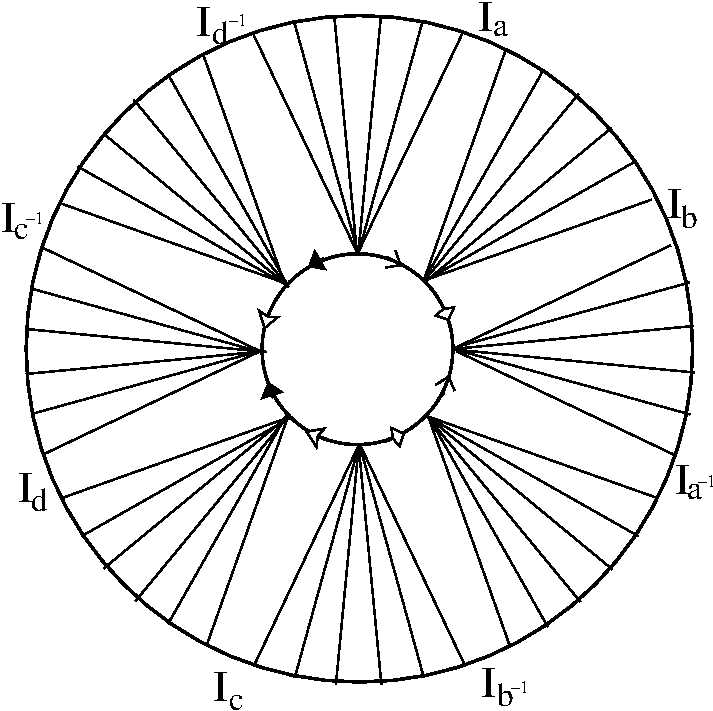}
\caption[]{This is a distorted view of the region $\cR$ in the hyperbolic plane together with all of the geodesics of the tesselation $\cT$ incident to $\cR$. Every interior angle incident to the inner circle in this diagram is $\pi/4$. There are 48 intervals in $\cI$ total. Only 8 special intervals are labeled.}\label{fig:surface}
\smallskip
\end{center}
\end{figure}


For $s\in S$, consider the edge $\cR \cap s\cR$. This edge is contained in a bi-infinite geodesic that separates the hyperbolic plane into two half-spaces. Let $L(s)$ denote the open arc of $\partial \D$ bounding the half space that contains $s\cR$. For each $I\in \cI$ let $s_I \in S$ be an element such that $I \subset L(s_I)$. For each $I$ there are either one or two choices for $s_I$. Define $f:\partial \D \to\partial \D$ by $f(x)=s_I^{-1}x$ for $x\in I$. As observed in \cite{bowen-series, series}, the map $f$ is Markov in the sense that for any $J \in \cI$, $f(I) \cap J \ne \emptyset$ implies $f(I) \supset J$. 

Let $V=\cI$, $E=\{ (I,J) \in V\times V:~f(I) \supset J\}$, $G=(V,E)$ be the associated directed graph, $\F=\langle a,b,c,d\rangle$ be the rank 4 free group, and $\cL:V \to \F$ be the map $\cL(I)=s_I$. We extend $\cL$ to the set of all finite directed paths in $G$ as explained in the introduction. In \cite[Theorem 5.10 and Corollary 5.11]{series} (see also \cite[Theorem 2.8]{birman-series}), the following is proven:
\begin{lem}\label{lem:series}
Let $\pi:\F \to \Lambda$ be the canonical surjection $\pi(s)=s$ for $s\in S$. Then for every $g\in \Lambda - \{e\}$ there is a unique element $w \in \F$ such that (a) $\pi(w)=g$ and (b) there exists some directed path $p$ in $G$ such that $\cL(p)=w$. Moreover, the word length of $w$ is the word length of $g$.
\end{lem}

\begin{thm}\label{thm:surface}
If $\Pi=(\Pi_{v,w})_{v,w\in V}$ is any stochastic matrix with $\Pi_{v,w}>0$ for all $(w,v)\in E$ then it is admissible of order 1. 
\end{thm}

\begin{proof}
In \cite{BufSer}, it is shown that the adjacency matrix of $G$ is irreducible. Equivalently, $G$ is strongly connected. 

For $s\in S$, let $I_s \subset \cI=V$ be the unique interval contained in $\cL(s) \setminus \cup_{t\ne s} \cL(t)$. By direct inspection we see that for any $s,t \in S$, $(I_s,I_t) \in E$ if and only if: $t \ne s^{-1}$ and $I_t$ is not adjacent to $I_{s^{-1}}$. For example, there are directed edges from $I_a$ to $I_c,I_{c^{-1}},I_d$ and $I_{d^{-1}}$ but there are no directed edges from $I_a$ to $I_{a^{-1}}, I_b$ or $I_{b^{-1}}$. There is also a loop from $I_a$ to itself. So if $v=a$ then $\Gamma_v$ contains $\cL( I_a)=a$, $\cL(I_a,I_c)=ac$, $\cL(I_a,I_d,I_c)=adc$,  $\cL(I_a,I_d,I_b)=adb$. Since $a,ac,adc,adb$ generate $\F_4$, we have $\Gamma_v=\F_4$.

Let $u=w=I_a$, $p=(I_a), q=(I_c), p^*=(I_{a^{-1}})$, $q^*=(I_{c^{-1}})$. Then 
\begin{itemize}
\item $upw, uqw, pq^*p, qp^*q$ are directed paths in $G$;
\item $\cL(p^*) = \cL(p)^{-1}$, $\cL(q^*) = \cL(q)^{-1}$.
\end{itemize}
So $G$ contains a good subgraph of order 1. 
\end{proof}

Corollary \ref{cor:surface} follows immediately from Lemma \ref{lem:series} and Theorems \ref{thm:surface} and \ref{thm:main}.

\subsection{Outline of the argument}

We consider the {\em synchronous tail equivalence relation} $\cR_{sync}$ on $V^\N$ given by
$$\cR_{sync}=\{ (s,t) \in V^\N\times V^\N:~ \exists N ~(s_i=t_i ~\forall i\ge N) \}.$$
For a natural number $k>0$ we also consider the {\em $k$-step asynchronous tail equivalence relation} on $V^\N$ given by
$$\cR_k = \{ (s,t) \in V^\N \times V^\N:~ \exists p \in\Z, N \in \N~ (s_{pk+i} = t_{i} ~\forall i\ge N)\}.$$
Let $\sigma:V^\N \to V^\N$ denote the shift map $\sigma(s)_i = s_{i+1}.$ Observe that $\cR_k$ is generated by $\cR_{sync}$ and the orbit-equivalence relation of $\sigma^k$. So we have the following natural inclusions:
$$\cR_{sync} \subset \cR_k \subset \cR_1.$$
More generally, $\cR_k \subset \cR_d$ if $d \mid k$.  We also have a cocycle $\alpha:\cR_1 \to \FF$ defined by
$$\alpha(s,t) = \cL(s_1)\cdots \cL(s_{N+p}) \cdot (\cL(t_1)\cdots \cL(t_N))^{-1}$$
where $N,p$ are such that $s_{p+i} = t_{i} ~\forall i\ge N$. 

Given a measure-preserving action $\FF \cc (X,\mu)$ on a probability space and a subequivalence relation $\cS$ of $\cR_1$, we let $\cS^X$ denote the skew-product equivalence relation on $V^\N \times X$:
$$\cS^X = \Big\{ \big((s,x), (t,y)\big):~ s\cS t, \alpha(t,s)x = y\Big\}.$$




Given a subequivalence relation $\cS \subset \cR_1$, let $\cF^X_\cS$ denote the sigma-algebra of measurable subsets of $V^\N \times X$ that are unions of $\cS^X$-equivalence classes. In other words, $\cF^X_\cS$ is the $\cS^X$-invariant sigma-algebra. 


For convenience, we will let $\cF_{sync}^X, \cF_k^X$ denote the $\cR_{sync}^X$ and $\cR_k^X$-invariant sigma-algebras respectively. The main technical step in the proof of Theorem \ref{thm:main} is:

\begin{thm}\label{thm:erg}
If the directed graph $G$ contains a good subgraph (as in Definition \ref{defn:good-graph}) then $\cF_{2k}^X = \cF_{sync}^X$ (up to sets of measure zero).
\end{thm}

We prove this in the next section and in \S \ref{sec:main} use it to prove Theorem \ref{thm:main}.

\subsection{Acknowledgements.}
The authors are deeply grateful to Vadim Kaimanovich for useful discussions.
Lewis Bowen is supported in part by NSF grant DMS-0968762, NSF CAREER
    Award DMS-0954606 and BSF grant 2008274.
Alexander Bufetov's research is carried out  thanks to the support of the A*MIDEX project (no. ANR-11-IDEX-0001-02) funded by the  programme ``Investissements d'Avenir '' of the Government of the French Republic, managed by the French National Research Agency (ANR). Bufetov is also supported in part by  the Grant MD-2859.2014.1 of the President of the Russian Federation,
by the Programme ``Dynamical systems and mathematical control theory''
of the Presidium of the Russian Academy of Sciences, by the ANR under the project ``VALET'' of the Programme JCJC SIMI 1,
and by the
RFBR grants 12-01-31284, 12-01-33020, 13-01-12449.

\section{Proof of Theorem \ref{thm:erg}}
Let $u,w\in V$ and $p,q, p^*,q^*$ be directed paths in $G$ satisfying the requirements of Definition \ref{defn:good-graph}. We need more notation:
\begin{notation}
If $s \in V^\N$ and $n<m$ are natural numbers then we let $s_{[n,m]} = (s_n,s_{n+1},\ldots, s_m) \in V^{m-n+1}$. We also write $s_{[n,\infty)} = (s_n,s_{n+1},\ldots) \in V^\N$. 
\end{notation}
Let us define
\begin{itemize}
\item $\tau_n:V^\N \to \N$ so that $\tau_n(s)$ is the $n$-th time of occurrence of either $upq$ or $uqw$. In other words, $\tau_n(s)$ is the smallest natural number so that there exist $i_1<i_2<\ldots< i_n$ with $i_n=\tau_n(s)$ so that for each $j$
$$s_{[i_j,i_{j}+k+1]} \in \{upw, uqw\}.$$


\item $\omega_n:V^\N \to V^\N$ by 

\begin{displaymath}
\omega_n(s)=\left\{ \begin{array}{cc}
s_{[1,\tau_n(s)]}qs_{[\tau_n(s)+k+1,\infty)} & \textrm{ if } s_{[\tau_n(s), \tau_n(s)+k+1]} = upw  \\
s_{[1,\tau_n(s)]}ps_{[\tau_n(s)+k+1,\infty)} & \textrm{ if } s_{[\tau_n(s), \tau_n(s)+k+1]} = uqw  
\end{array}
\right.
\end{displaymath}

\item Note that $\omega_n$ is invertible. So we can define $\psi_n: V^\N \to V^\N$ by
\begin{displaymath}
(\psi_n\omega_n(s))=\left\{ \begin{array}{cc}
\omega_n(s)_{[2k+1,\tau_n(s)+k]}p^*\omega_n(s)_{[\tau_n(s)+1,\infty)} & \textrm{ if }s_{[\tau_n(s), \tau_n(s)+k+1]} = upw \\
\omega_n(s)_{[2k+1,\tau_n(s)+k]}q^*\omega_n(s)_{[\tau_n(s)+1,\infty)} & \textrm{ if }s_{[\tau_n(s), \tau_n(s)+k+1]} = uqw \\
\end{array}
\right. 
\end{displaymath}
\begin{displaymath}
\quad ~~=\left\{ \begin{array}{cc}
s_{[2k+1,\tau_n(s)]}qp^*qs_{[\tau_n(s)+k+1,\infty)} & \textrm{ if }s_{[\tau_n(s), \tau_n(s)+k+1]} = upw \\
s_{[2k+1,\tau_n(s)]}pq^*ps_{[\tau_n(s)+k+1,\infty)} & \textrm{ if }s_{[\tau_n(s), \tau_n(s)+k+1]} = uqw \\
\end{array}
\right.
\end{displaymath}

\item Recall that $\nu$ is the $\Pi$-stationary measure on $V$. Let $\tilde{\nu}$ be the associated measure on $V^\N$. To be precise, for any $t_1,\ldots, t_n \in V$,
$$\tilde{\nu}( \{s\in V^\N:~ s_i=t_i~\forall 1\le i \le n\}) = \nu(t_n)\Pi_t = \nu(t_n) \Pi_{t_n,t_{n-1}}\cdots \Pi_{t_2,t_1}.$$ 

\item $C>0$ be a constant so that almost everywhere holds
$$C^{-1} \le \frac{ d (\omega_n^{-1})_*\tnu}{d\tnu}(s) \le C, \quad C^{-1} \le \frac{ d ((\psi_n\omega_n)^{-1})_*\tnu}{d\tnu}(s) \le C$$
The existence of such a constant follows from the finiteness of $V$ (so that there is a uniform bound on the ratio of any two nonzero entries of $\Pi$) and explicit computation using the formulae above.
 

\end{itemize}


Recall that $\sigma:V^\N \to V^\N$ is defined by $\sigma(s)_i = s_{i+1}$. Let $d_{V^\N}$ denote the distance function on $V^\N$ defined by $d_{V^\N}\big( (s_1,s_2, \ldots), (t_1, t_2, \ldots) \big) = \frac{1}{n}$ where $n$ is the largest  natural number such that $s_i=t_i$ for all $i < n$.
\begin{prop}\label{prop:ve}
For every $n>2k+1$,
\begin{enumerate}
\item $\forall s \in V^\N$, $d_{V^\N}\big(\psi_n  \omega_n (s), \sigma^{2k} \omega_n(s)\big)\le \frac{1}{\tau_n(s)-k}$;
\item $\forall s \in V^\N$,  $d_{V^\N}(s,\omega_n s )\le \frac{1}{\tau_n(s)}$;
\item the graphs of $\omega_n$ and $\psi_n$ are contained in $\cR_{sync}
$;
\item $\forall s \in \cA^\N$, $\alpha(\psi_n\omega_n s, \omega_n s) = \alpha( \sigma^{2k}\omega_n s, s)$. 

\item $\forall f \in L^1(\cA^\N)$, $\norm{f\circ \omega_n}_1 \le C \norm{f}_1$ and $\norm{f \circ \psi_n}_1 \le C^2 \norm{f}_1$.  

\end{enumerate}
\end{prop}

\begin{proof}
Items 1 and 2 are obvious. It is clear that the graph of $\omega_n$ is contained in $\cR_{sync}$. This implies the graph of $\psi_n\omega_n$ is contained in $\cR_{sync}$ and therefore, since $\omega_n$ is invertible, the graph of $\psi_n$ is contained in $\cR_{sync}$. 

For simplicity's sake, we will drop the subscripts $n$ in the following computations. So $\psi=\psi_n,\omega=\omega_n,\tau=\tau_n$.

Suppose that $s \in V^\N$ satisfies $s_{[\tau(s),\tau(s)+k+1]} = upw$.  Let $N=\tau(s)$. Because $(\psi\omega(s))_i = \omega(s)_i$ for all $i>N$ the definition of $\alpha$ implies 
\begin{eqnarray*}
\alpha(\psi\omega s, \omega s)
&=& \cL(\psi\omega(s)_1) \cdots \cL(\psi\omega(s)_{N})\Big( \cL(\omega(s)_1) \cdots \cL(\omega(s)_N)\Big)^{-1}\\
&=& \cL(s_{1+2k})\cdots \cL(s_{N}) \cL(q_1)\cdots \cL(q_k) \cL(p_k)^{-1}\cdots \cL(p_1)^{-1}\Big( \cL(s_1) \cdots \cL(s_N) \Big)^{-1}\\
\end{eqnarray*}
Because $(\sigma^{2k} \omega s)_{i-2k} = (\omega s)_i = s_i$ for all $i > N+k$ the definition of $\alpha$ implies 
\begin{eqnarray*}
\alpha( \sigma^{2k}\omega s, s)&=& \cL(\sigma^{2k}\omega(s)_1) \cdots \cL(\sigma^{2k}\omega(s)_{N-k})\Big( \cL(s_1)\cdots \cL(s_{N+k}) \Big)^{-1} \\
&=& \cL((\omega s)_{1+2k}) \cdots \cL((\omega s)_{N+k})\Big(\cL(s_1)\cdots \cL(s_{N+k})\Big)^{-1} \\
&=& \cL(s_{1+2k}) \cdots \cL(s_N) \cL(q_1)\cdots \cL(q_k)  \Big( \cL(s_1)\cdots \cL(s_N) \cL(p_1)\cdots \cL(p_k)  \Big)^{-1} \\
&=& \alpha(\psi\omega s, \omega s).
\end{eqnarray*}
The case when $s_{[\tau(s),\tau(s)+k+1]} = uqw$ is similar. This proves item 4.

It follows from the choice of $C>0$ (made right before this proposition) that for every $f\in L^1(V^\N)$,
 $$\norm{f\circ \omega}_1 \le C \norm{f}_1, \quad \| f\circ \psi \omega \|_1 \le C \|f\|_1.$$
 Since $\omega$ is invertible, this implies 
$$\| f \circ \psi\|_1 = \| f\circ \psi \omega \circ \omega^{-1} \|_1 \le C \| f\circ \psi \omega  \|_1 \le C^2 \|f\|_1.$$
Here we used that $\omega=\omega^{-1}$. This establishes the last claim.
\end{proof}

\begin{defn}\label{defn:sigma}
Define $\sigma_X: V^\N \times X \to V^\N \times X$ by $\sigma_X(s,x) =( \sigma s, \alpha(\sigma s,s)x)$. Note $\alpha(\sigma s,s) = s_1^{-1}.$
So we can also write $\sigma_X(s,x) = (\sigma s,  s_1^{-1} x).$
\end{defn}

\begin{lem}
There exist measurable maps $\Phi_n,\Psi_n,\Omega_n:V^\N \times X \to V^\N \times X$ (for $n > 2k+1$) such that
\begin{enumerate}
\item for all $f\in L^1(V^\N \times X)$, $\lim_{n\to\infty} \|f\circ \Psi_n \circ \Omega_n- f\circ \sigma_X^{2k} \circ \Phi_n\|_1 =0$;
\item for all $f\in L^1(V^\N \times X)$, $\lim_{n\to\infty} \|f\circ \Omega_n-f\|_1=0$;
\item the graphs of $\Phi$ and $\Psi$ are contained in $\cR_{sync}
^X$.
\end{enumerate}
\end{lem}

\begin{proof}
For $n>2k+1$ an integer, let $\psi_n$ and $\omega_n$ be as in Proposition \ref{prop:ve}. Define
\begin{eqnarray*}
\Omega_n(s, x) &:=& (\omega_n  s, x) \\
\Phi_n(s, x)        &:=& (\omega_n  s, \alpha(\omega_n  s, s) x)\\
\Psi_n(s, x)        &:=& (\psi_n s, \alpha( \psi_n  s, s) x).
\end{eqnarray*}

Since the graphs of $\psi_n$ and $\omega_n$ are contained in $\cR_{sync}
$, the graphs of $\Phi_n$ and $\Psi_n$ are contained in $\cR_{sync}^X$. Let $d_X$ be a metric on $X$ that induces its Borel structure and makes $X$ into a compact space. For $(s, x),(s', x') \in V^\N \times X$, define $d_*((s, x),(s', x')) = d_X(x,x')  + d_{V^\N} ( s, s')$. By the previous proposition, $d_*(\Omega_n(s, x), (s, x) ) = d_{V^\N}(\omega_n(s), s) \le 1/\tau_n(s) \le 1/n$. Also by the previous proposition: 
\begin{eqnarray*}
\Psi_n\Omega_n(s, x) &=& (\psi_n\omega_n s, \alpha( \psi_n\omega_n  s, \omega_n  s) x) \\
\sigma_X^{2k}\Phi_n(s, x) &=& \sigma_X^{2k}( \omega_n s, \alpha(\omega_n  s,  s) x) = (\sigma^{2k}\omega_n s, \alpha(\sigma^{2k}\omega_n s,\omega_n s)\alpha(\omega_n s, s)x)\\
&=& ( \sigma^{2k}\omega_n s, \alpha(\sigma^{2k}\omega_n s, s)x) = (\sigma^{2k}\omega_n s, \alpha( \psi_n\omega_n  s, \omega_n  s) x).
\end{eqnarray*}
So the previous proposition implies $d_*(\Psi_n \circ \Omega_n(s, x), \sigma_X^{2k} \circ \Phi_n(s, x)) \le 1/(n-k)$.  So if $f$ is a continuous function on $V^\N \times X$ then the bounded convergence theorem implies
\begin{eqnarray*}
\lim_{n\to\infty} \|f\circ \Psi_n \circ \Omega_n- f\circ \sigma_X^{2k} \circ \Phi_n\|_1 &=&0\\
\lim_{n\to\infty} \|f\circ \Omega_n-f\|_1&=&0.
\end{eqnarray*}
It follows from the previous proposition that the operators $f \mapsto f \circ \Omega_n$, $f\mapsto f \circ \Phi_n$ and $f\mapsto f \circ \Psi_n$ are all bounded for $f \in L^1(V^\N \times X)$ with bound independent of $n$. It easy to see that $f \mapsto f\circ \sigma_X^{2k}$ is also a bounded operator on $L^1(V^\N \times X)$ (because $V$ is finite and $\tilde{\nu}$ is the Markov measure). Since the continuous functions are dense in $L^1(V^\N \times X)$, this implies the lemma.
\end{proof}

We can now prove Theorem \ref{thm:erg}.

\begin{proof}[Proof of Theorem \ref{thm:erg}]
Note $\cF^X_{2k} \supset \cF_{sync}^X$. So it suffices to show that if $f\in L^1(V^\N\times X)$ is $\cR^X_{sync}$-invariant then it is $\cR^X_{2k}$-invariant. 
Because the map $\sigma^{2k}_X$ together with $\cR^X_{sync}$ generates $\cR^X_{2k}$, it suffices to show that  if $f \in L^1(V^\N\times X)$ is $\cR^X_{sync}$-invariant then $f \circ \sigma_X^{2k} = f$. 

 Let $\Phi_n, \Psi_n, \Omega_n$ ($n>2k+1$) be as in the previous lemma. Because $f$ is  $\cR^X_{sync}$-invariant and the graph of $\Psi_n$ is contained in $\cR_{sync}
^X$, it follows that $f\circ \Psi_n=f$ for all $n$. An easy exercise shows that $\sigma_X$ preserves the equivalence relation in the sense that
$$\Big( (s,x), (t,y) \Big) \in \cR_{sync} \Rightarrow \Big( \sigma_X(s,x), \sigma_X(t,y) \Big)\in \cR_{sync}.$$
It follows that $f\circ \sigma_X^{2k}$ is $\cR^X_{sync}$-invariant. Since the graph of $\Phi_n$ is contained in $\cR^X_{sync}$,  $f\circ\sigma_X^{2k}\circ \Phi_n = f\circ \sigma_X^{2k}$ for all $n$. We now have
\begin{eqnarray*}
\|f - f\circ\sigma_X^{2k}\|_1 &=& \|f - f\circ\sigma_X^{2k} \circ \Phi_n\|_1\\
&\le& \|f - f \circ \Psi_n \circ \Omega_n\|_1 +\|f \circ \Psi_n \circ \Omega_n-f\circ\sigma_X^{2k} \circ \Phi_n\|_1\\
& =& \| f-f\circ\Omega_n\|_1+\|f \circ \Psi_n \circ \Omega_n-f\circ\sigma_X^{2k} \circ \Phi_n\|_1.
\end{eqnarray*}
We take the limit as $n\to\infty$ (using the previous lemma) to obtain  $f=f\circ \sigma_X^{2k}$ as claimed.

\end{proof}

\section{Proof of Theorem \ref{thm:main}} \label{sec:main}

\begin{prop}\label{prop:kakutani}
Let $\Pi,V,\cL$ be as above. For each $v \in V$, let $\Gamma_v \le\FF$ be the subgroup generated by all elements of the form $\cL(p)$ where $pv$ is a directed path from $v$ to $v$ in $G$. If $\Gamma_v=\FF$ for some $v\in V$ and $G$ is strongly connected then $\cF_1^X$ is the $\sigma$-algebra generated by all sets of the form $V^\N \times A$ where $A \subset X$ is a measurable $\FF$-invariant set. In particular, if $\FF \cc (X,\mu)$ is ergodic then $\cF_1^X$ is trivial.
\end{prop}

\begin{proof}
By decomposing into ergodic components, we may assume that $\FF \cc (X,\mu)$ is ergodic. Because $\cR_1^X$ is generated by $\sigma_X$, it suffices to prove $\sigma_X$ is ergodic.

Let $Y \subset V^\N \times X$ be the set of all $(s,x)$ such that $s_1=v$ where $v\in V$ is chosen so that $\Gamma_v=\FF$. Let $T:Y \to Y$ be the induced transformation:
$$T(s,v)= \sigma_X^n(s,v)$$
where $n\ge 1$ is the smallest natural number such that $\sigma^n_X(s,v) \in Y$. By Kakutani's random ergodic theorem \cite[Theorem 3 (a) $\Rightarrow$ (f)]{Kakutani}, the ergodicity of $\FF \cc (X,\mu)$ implies $T$ is ergodic. 

Now suppose $Z \subset V^\N \times X$ is measurable, $\sigma_X$-invariant and has positive measure. Then $Y \cap Z$ is $T$-invariant. Because the graph $G$ is strongly connected, $\tilde{\nu}\times \mu(Y\cap Z)>0$. Since $T$ is ergodic, this implies $Y \cap Z =Y$ up to measure zero. However, $\cup_{i=0}^\infty \sigma_X^i Y = V^\N \times X$ (up to measure zero) because $G$ is strongly connected. This implies $Z$ is conull and therefore $\sigma_X$ is ergodic as claimed.
\end{proof}

\begin{lem}\label{lem:period}
For any $f\in L^1(V^\N\times X)$ and any $k \in \N$,
$$\frac{1}{k}\sum_{i=0}^{k-1} \EE[f  \circ \sigma^i_X |\cF^X_k] = \EE[f |\cF^X_1].$$
\end{lem}

\begin{proof}
Because $\cF^X_1$ is the sigma-algebra of $\sigma_X$-invariant measurable subsets, von Neumann's mean ergodic theorem implies that
$$ \frac{1}{nk} \sum_{i=0}^{nk-1} f \circ \sigma^i_X \to \EE[f | \cF^X_1]$$
in $L^1$ as $n \to \infty$. By taking conditional expectations on both sides (and remembering that $\cF^X_1 \subset \cF^X_k$), we have 
$$ \frac{1}{nk} \sum_{i=0}^{nk-1} \EE[f \circ \sigma^i_X|\cF^X_k] \to \EE[f | \cF^X_1].$$
Because $\cF^X_k$ is $\sigma_X^k$-invariant, we have $\EE[f \circ \sigma^{k+i}_X|\cF^X_k] =\EE[f \circ \sigma^i_X|\cF^X_k]$ for any $i$. So for any $n$
$$\frac{1}{nk} \sum_{i=0}^{nk-1} \EE[f \circ \sigma^i_X|\cF^X_k] =\frac{1}{k} \sum_{i=0}^{k-1} \EE[f \circ \sigma^i_X|\cF^X_k].$$
This implies the lemma.
\end{proof}




\begin{proof}[Proof of Theorem \ref{thm:main} from Theorem \ref{thm:erg}]
Without loss of generality, we may assume $\FF \cc (X,\mu)$ is ergodic. Let $\pi:V^\N \times X \to V \times X$ denote the projection map $\pi(s,x)=(s_1,x)$. 


Let $\cB_{V\times X}$ denote the Borel sigma-algebra on $V\times X$ and let $\cF^X_{\ge n}$ be the smallest sigma-algebra of $V^\N \times X$ containing $(\pi \circ \sigma_X^m)^{-1}(\cB_{V\times X})$ for every $m\ge n$.

Consider the induced Markov operator $\Pi_X:L^1(V\times X) \to L^1(V\times X)$ given by
$$\Pi_X(\varphi)(x,v) = \sum_{w\in V} \Pi_{w,v}\varphi(w,T_v x).$$
Observe that for $n\ge 2$
\begin{eqnarray*}
\Pi^n_X(\varphi)(x,v) &=&\sum_{t_1,\ldots, t_n \in V} \Pi_{(t_1,\ldots, t_n,v)} \varphi( t_1, T_{(t_2,\ldots, t_n,v)}x).
\end{eqnarray*}
Thus
\begin{eqnarray*}
(\Pi^n_X \varphi)\circ \pi \circ \sigma_X^n)(s,x) &=& (\Pi^n_X\varphi)( s_{n+1}, T_{(s_1,\ldots, s_n)}^{-1}x) \\
&=& \sum_{t_1,\ldots, t_n \in V} \Pi_{(t_1,\ldots, t_n, s_{n+1})} \varphi( t_1, T_{(t_2, \ldots, t_n,s_{n+1})}T_{(s_1,\ldots, s_n)}^{-1}x) \\
&=& \EE[\varphi  \pi |\cF^X_{\ge n+1}](s,x). 
\end{eqnarray*}
The reverse martingale convergence theorem yields
$$\EE[\varphi  \pi |\cF^X_{\ge n+1}] \to \EE[\varphi \pi | \cF^X_{sync}]$$
in $L^1(V^\N \times X)$ as $n\to\infty$.  By Theorem \ref{thm:erg}, $\cF^X_{sync} = \cF^X_{2k}$. Therefore, 
$$(\Pi^n_X \varphi)\circ \pi \circ \sigma_X^n \to \EE[\varphi \pi | \cF^X_{2k}]$$
in $L^1(V^\N \times X)$ as $n\to\infty$. Because conditioning on $\cF^X_{2k}$ commutes with $\sigma_X$, for any $i \ge 0$
$$(\Pi^n_X \varphi)\circ \pi \circ \sigma_X^{n+i} \to \EE[\varphi \pi \circ \sigma_X^i| \cF^X_{2k}]$$
in $L^1(V^\N \times X)$ as $n\to\infty$. Since $\EE[\varphi \pi \circ \sigma_X^i| \cF^X_{2k}] = \EE[\varphi \pi \circ \sigma_X^{2k+i}| \cF^X_{2k}]$ we can also write this as: for any $0\le i < 2k$,
$$(\Pi^n_X \varphi)\circ \pi \circ \sigma_X^{n-i} \to \EE[\varphi \pi \circ \sigma_X^{2k-i}| \cF^X_{2k}]$$
in $L^1(V^\N \times X)$ as $n\to\infty$. Now Lemma \ref{lem:period} and Proposition \ref{prop:kakutani} imply
\begin{eqnarray*}
 \frac{1}{2k}\sum_{i=0}^{2k-1} (\Pi^n_X \varphi)\circ \pi \circ \sigma_X^{n-i}
 &\to & \frac{1}{2k}\sum_{i=0}^{2k-1} \EE[\varphi \pi \circ \sigma_X^i | \cF^X_{2k}] = \EE[\varphi \pi | \cF^X_{1}] =  \int \varphi~d\nu \times \mu
\end{eqnarray*}
in $L^1$ as $n\to\infty$.  However,
$$(\Pi^n_X \varphi)\circ \pi \circ \sigma_X^{n-i} = (\Pi^{n-i}_X \Pi^i_X\varphi)\circ \pi \circ \sigma_X^{n-i} \to \EE[\Pi^i_X\varphi \pi | \cF^X_{2k}]$$
in $L^1(V^\N \times X)$ as $n\to\infty$. Similarly, 
$$(\Pi^{n+i}_X \varphi)\circ \pi \circ \sigma^n_X \to \Pi^i_X\Big(\EE[\varphi \pi | \cF^X_{2k}]\Big)= \EE[\Pi^i_X\varphi \pi | \cF^X_{2k}]$$
in $L^1(V^\N \times X)$ as $n\to\infty$. So we have
$$ \frac{1}{2k}\sum_{i=0}^{2k-1} (\Pi^{n+i}_X \varphi)\circ \pi \circ \sigma^n_X \to \int \varphi~d\nu \times \mu$$
in $L^1$ as $n\to\infty$.

Without loss of generality, we may assume $\int \varphi~d\nu \times \mu = 0$ in which case the above implies
\begin{eqnarray*}
\left\|\frac{1}{2k}\sum_{i=0}^{2k-1} (\Pi^{n+i}_X \varphi)\circ \pi \circ \sigma^n_X \right\| \to 0
\end{eqnarray*}
as $n\to\infty$. However,
\begin{eqnarray*}
\left\| \frac{1}{2k}\sum_{i=0}^{2k-1} (\Pi^{n+i}_X \varphi)\circ \pi \circ \sigma^n_X \right\| &=&\left\| \frac{1}{2k}\sum_{i=0}^{2k-1} \Pi^{n+i}_X \varphi \right\|.
\end{eqnarray*}
So
$$\frac{1}{2k}\sum_{i=0}^{2k-1} \Pi^{n+i}_X \varphi \to 0$$
in $L^1$ as $n\to\infty$. Next we note that if $\varphi(v,x)=\phi(x)$ for some $\phi \in L^1(X)$ then by a change of variables argument
\begin{eqnarray*}
(S_n \phi)(x) &=& \sum_{s_1,\ldots, s_n \in V} \nu(s_n) \Pi_{(s_1,\ldots, s_n)}  \phi( T_{(s_1,\ldots, s_n)} x)\\
&=&   \sum_{v\in V} \sum_{s_1,\ldots, s_{n-1} \in V} \nu(v) \Pi_{(s_1,\ldots, s_{n-1},v)}   \phi( T_{(s_2,\ldots, s_{n-1},v)} x) \\
&=& \sum_{v\in V} \nu(v) (\Pi^{n-1}_X \varphi)(v,x).
\end{eqnarray*}
Thus $S_n \phi$ converges to $0$ in $L^1$ as $n\to\infty$.

\end{proof}

\end{document}